\documentclass[12pt,reqno]{amsart} 
\usepackage{amsmath,amsfonts,latexsym,amssymb}
\usepackage{graphicx,subfig}
\usepackage{enumerate} 
\usepackage[colorlinks=true, pdfstartview=FitV,
linkcolor=cyan, citecolor=red, urlcolor=blue]{hyperref}
\usepackage[utf8]{inputenc} 
\usepackage[T1]{fontenc}
\usepackage{ae,aecompl} 
\usepackage{dsfont} 
\usepackage{pxfonts}
\usepackage{microtype} 
\usepackage{braket}

\newtheorem{theorem}{Theorem}[subsection]

\newtheorem{proposition}[theorem]{Proposition}
\newtheorem{corollary}[theorem]{Corollary}

\newtheorem{definition}[theorem]{Definition}

\renewcommand{\leq}{\leqslant}
\renewcommand{\geq}{\geqslant}
\newcommand{\bb}{\mathsf b}
\newcommand{\mc}{\mathcal}

\newcommand{\R}{\mathcal  R}
\renewcommand{\L}{\mathcal L}
\newcommand{\p}{\mathfrak p}
\newcommand{\q}{\mathfrak q}
\newcommand{\F}{\mathfrak F}
\newcommand{\G}{\operatorname{Gap}}
\newcommand{\Fd}{{\mathsf F}_2^+}

\renewcommand{\leq}{\leqslant} 
\renewcommand{\geq}{\geqslant}
\renewcommand{\epsilon}{\varepsilon}

\newcommand{\mk}{\mathfrak} 
\newcommand{\grf}{\pi_1(\Sigma)}
\newcommand{\bgrf}{\partial_{\infty}\grf}
\newcommand\defeq
{{}\mathrel{\mathop:}={}} 

\newcommand{\E}{\operatorname{Error}}
\newcommand{\demi}{\frac{1}{2}}

\title[Probability and McShane's identity]{The probabilistic nature of McShane's identity: planar tree coding of simple loops}

\author[Labourie]{François LABOURIE}
\address{Universit\'e C\^ote d'Azur, LJAD, Nice F-06000; FRANCE}
\author[Tan]{Ser Peow TAN}
\address{National University of Singapore; SINGAPORE}
\thanks{The research leading to these results has received funding from the European Research Council under the {\em European Community}'s seventh Framework Programme (FP7/2007-2013)/ERC {\em grant agreement} ${\rm n^o}$ FP7-246918. Tan was partially supported by the National University of Singapore academic research grant R-146-000-235-114.}
\begin{document}
\maketitle
\begin{center}
{\em With admiration, to our friend and mentor: Bill Goldman.}
\end{center}
\tableofcontents
\section{Introduction}

In his celebrated paper \cite{McShane:1998}, Greg McShane showed that for a  1-cusped hyperbolic torus the following identity holds:
\begin{eqnarray}
1=\sum_{\gamma\in\mathcal G}
\frac
{1}
{e^{\ell(\gamma)}+1}\ ,
\end{eqnarray}
where the sum runs over the set $\mathcal G$ of oriented  simple loops in the torus. Similarly for 1-cusped hyperbolic surfaces
\begin{eqnarray}
1=\sum_{P\in\mathcal P}
\frac{1}{e^{\frac{\ell(\partial P)}{2}}+1}\ .
\end{eqnarray}
where the sum runs now over the space $\mathcal P$ of embedded 1-cusp pair of pants -- or equivalently the space of oriented simple arcs from the cusp to itself. 

This formula has had a lot of descendants and generalisations, to quasifuchsian spaces \cite{Tan:2004wz,Tan:2004,Tan:2005fr,Tan:2008vv,Tan:2015wr}, to higher rank geometries \cite{Labourie:2009wi}, to surfaces with boundary components \cite{Mirzakhani:2007tt}. An extensive survey of geometric identities with similar features can be found in \cite{Bridgeman:2016tw}. Quite notably, this identity has played a fundamental role in Mirzakhani acclaimed proof of the Kontsevich--Witten formula for intersection numbers \cite{Mirzakhani:2007vi}.

Among these proofs stands a remarkable argument by Bowditch \cite{Bowditch:1996} proving McShane's identity on the torus using Markoff triples and a combinatorial approach. In turn this interpretation has given rise to yet another higher rank construction \cite{Tan:3c1rgtLw} whose geometrical interpretation remains elusive.

The present paper grow out of the fascination of both authors for Bowditch's proof and wishes to present a generalisation of Bowditch's proof for higher genus surfaces, equally working in the context of higher geometries. More importantly, we want to present a new interpretation of McShane's identity: essentially, we show that this formula has a probabilistic nature -- as emphasized by the 1 on the lefthand side of the equations above --  and we hope that in the future this could lead to a better understanding of Mirzakhani's topological recursion.

Essentially our interpretation runs as follows. Let $S$ be a topological surface with one marked point $x_0$. As a rebranding of a Lemma by  L. Mosher \cite{Mosher:1995dw}, we show that given any triangulation with $x_0$ as  the only vertex, one can ``code" any simple arc in $S$ passing through $x_0$  by  an edge of the original triangulation together with a word in the semigroup with two generators (Theorem \ref{theo:topcod}) and  similar results holds with infinite words and lamination (Theorem \ref{theo:exten}). We actually consider these words as embedded paths in some trivalent (except at the root) rooted planar tree $\mc T$.

Moreover, in that context the planar structure of the tree allows us to distinguish  between {\em rational paths} -- which end up coding rational laminations -- and {\em irrational paths} -- which interpret irrational laminations.

Now the geometry -- a hyperbolic structure or more generally a cross ratio \cite{Labourie:2009wi}\cite{Labourie:2005} -- gives rise to a harmonic 1-form on $\mc F$ and thus to a probability measure $\mu$ on the space of embedded paths.

Finally, using the Birman--Series Theorem \cite{Birman:1985} as in \cite{McShane:1998}, one concludes that the measure of the set of irrational paths is  zero, whereas the measure on the space of rational paths corresponds to the right-hand side of the equations above. Since the rational paths are countable and essentially labelled by simple paths though $x_0$ thanks to our coding theorem, McShane's identity is just expressing that  the sum of the measure of rational paths is one.

Although, there are no new results in this article and in many ways the proofs are either trivial or well known, we hope that our presentation insisting on the probabilistic nature of McShane's identity will be helpful in the future.

The first section recalls harmonic measures on trees and how they define random paths in section \ref{sec:harm}. Then we specialize to planar trees in section \ref{sec:plan} and describe the gap (the measure of rational paths) and error terms (the measure of irrational paths) of a general harmonic measure  in Section \ref{sec:gap}. In Section \ref{sec:err}, we relate the error term to the Lebesgue measure of a certain set on the circle.  In section \ref{sec:tree}, we introduce  a path coding geodesics once one choses a triangulation. Then in section \ref{sec:cr}, we explain how hyperbolic structures and more generally cross ratios give rise to specific measures. In section \ref{sec:BS}, we show that the Birman--Series Theorem implies that the error term vanishes. Finally in Section \ref{sec:codBD}, we indicate how one could generalize the above construction to surfaces with more boundary components.

We thank François Ledrappier, Maryam Mirzakhani, Hugo Parlier and Saul Schleimer for helpful discussions.

\section{Green formula for rooted trees}\label{sec:harm}

We will recall in this section very standard material about harmonic analysis on graphs.

Let $\mc T$ be a rooted tree whose root is denoted $v_0$. Motivated  by the applications to geodesics, we shall assume that all vertices  --except for the root -- has valence 3. 

 Every edge comes with an orientation: we say an edge is {\em positively oriented} or {\em positive} if it comes outward from the root, negatively oriented if it goes inward. We shall denote by $S(n)$ the sphere of combinatorial radius $n$ centred at the root $v_0$.  If $v$ is a vertex, let $\vert v\vert$ be the distance to the root, if $e$ is a positive edge we define by $\vert e\vert\defeq\vert v^+\vert$, where $v^+$ is the endpoint of $e$. Finally, if $e$ is an edge we denote by $\bar e$ the edge with the opposite orientation.

\subsection{Harmonic forms}

A  function $\Phi$ on the set of edges of the tree is said to be a {\em 1-form} if
$$
\Phi(\bar e)=-\Phi(e),
$$
moreover $\Phi$ is  {\em a harmonic}, if for every vertex different from the root, if $e_x$ is the incoming edge (from the root) and $e_\pm$ the outgoing edges, we have
\begin{align}
  \Phi(e_x)=\Phi(e_+)+\Phi(e_-).
\end{align}
Since $\Phi$ is a function of the edges, we could consider it  as a 1-form on the graph. Then the above equation just says that $\Phi$ considered as a 1-form is harmonic, everywhere except possibly at the root. Let us  then define
\begin{align}
  \partial\Phi\defeq\sum_{e\in S(1)}\Phi(e).
\end{align}

We have
\begin{corollary}{\sc[Green formula]}
The following equality holds
\begin{align}
\sum_{e\in S(n)}\Phi(e)&=\partial\Phi.
\end{align}
\end{corollary}

\subsection{The harmonic measure}   
Assume now that $\Phi$ is positive on positive edges.
Any vertex $x$ in the tree other than the root  is trivalent, let us define a probability measure $\mu_x$ on the edges having $x$ as an extremities (two outgoing, one ingoing) in the following way
\begin{itemize}
\item the probability $\mu_x(e_x)=0$ if $e_x$ is the edge directed towards the root,
\item Otherwise, let 
$$
\mu_x(e)=\frac{\Phi(e)}{\Phi(e_x)}.
$$
\end{itemize}
This family of measures defines an inhomogeneous Markov process. Let $\mc P$ be the set of infinite embedded paths in the tree starting from the root. Although we shall not need it we may identify $\mc P$ with $\partial_\infty\mc T$. Let $\pi_n$ be the map from $\mc P$ to set of edges of $\mc T$ which associates to a path its $n$ step:
$$
\p^n\defeq\pi_n(\p)\in S(n)\ .
$$

 We thus have the {\em harmonic measure}, see \cite{Ledrappier:2001uv}, $\mu_\Phi$ on $\mc P$ associated to this random process. By definition, this measure is such that 
\begin{align}
  \pi^n_*(\mu_\Phi)=\left.\Phi\right\vert_{S(n)}.
\end{align} 
In particular
\begin{align}
\mu_\Phi\{\p\}=\Phi_\infty(\p)\  \hbox{ and } \mu_\Phi(\mc P)=\partial\Phi\ ,\label{eq:atom}
\end{align}
where, given a path $\p$, we define 
$$
\Phi_\infty(\p)=\liminf_{n\to\infty}\Phi(\pi_n(\p)).
$$
Observe that since  $\Phi$ is positive on positive edges -- that is edges oriented away from the root -- $\Phi(\pi_n(\p))$ is decreasing as a function of $n$ and thus $\Phi_\infty$ is actually the limit of $\Phi\circ \pi_n$.

By construction, if $e\in S(n)$,  $\Phi(e)$ is the probability that a random path   starting from the root arrives at $e$ at the  $n$ step.

\section{Planar tree, complementary regions and rational paths}\label{sec:plan}

A {\em planar tree} is a tree that can be embedded in the plane. Equivalently a planar structure on the tree is given by a cyclic order on the edges outgoing form each vertex. In 
particular, 
\begin{itemize}
	\item we obtain a cyclic order on $S(n)$,
	\item  Given a positive edge $e$ arriving in $v$, we define the edges $\L(e)$ and $\R(e)$ to be the outgoing edges from $v$ so that $(\bar e, \L(e),\R(e))$ is positively oriented. In particular, the edges $\L(e)$ and $\R(e)$ are positive.
\end{itemize}

\subsection{Complementary regions}

Every  edge  $e$ in the tree defines (injectively)  a  {\em complementary region} of the planar graph if we considered it properly embedded in the plane. This region, also denoted  $e$, is bounded by the edges $\L^n\R(e)$ and $\R^n\L(e)$ for all non negative integers $n$ and we define the paths
$$
\partial^Re\defeq \bigcup_{n\in\mathbb N}\L^n\R(e),\ \  \partial^L e\defeq \bigcup_{n\in\mathbb N}\R^n\L(e)\ 
$$
where the paths are completed by adding the necessary initial edges so that they start at $v_0$ (by abuse of notation we will usually consider this tail end of the path to be the path).

So far some complementary regions are missing from this construction, namely those innermost regions which are adjacent to the root $v_0$. Let us label those as well. Let  $e^1,\ldots,e^N$ are the edges stemming from the root $v_0$ in their  cyclic order, let by definition 
$S(0)\defeq \{(e^i,e^{i+1})\mid i\in\{1,\ldots\}\}$. Then any $f=(e^i,e^{i+1})$ defines also a {\em complementary region} bounded by the two paths $\R^n(e^i)$ and $\L^n(e^{i+1})$. By convention we write $\L f\defeq e^i$ and $\R f\defeq e^{i+1}$. Then, similarly 
$$
\partial^R f\defeq \bigcup_{n\in\mathbb N}\L^n\R(f),\ \  \partial^L f\defeq \bigcup_{n\in\mathbb N}\R^n\L(f)\ .$$

The set 
$$
\mc F\defeq \bigsqcup_{n\in\mathbb N}S(n)\ ,
$$
is the set of {\em complementary regions}.

\subsection{Rational paths} \label{sec:rat-path}

The planar structure on the tree helps us to distinguish between ``irrational" and ``rational" paths: we saw that every complementary region  $e$ defines two paths $\partial^Le$ and $\partial^R e$ starting from $v_0$. By definition we call these paths {\em rational} and any other paths {\em irrational}. Observe that the set $\mc Q$  of rational paths is countable.

Remark that the set of paths $\mc P$ inherits a lexicographic order from the order in $S(n)$. 
If $\p_0$ and $\p_1$ are two paths with $p_0<p_1$ in the lexicographic order, we define
\begin{eqnarray}
[\p_0,\p_1]&\defeq&\{\p\mid p_0\leq p\leq p_1\}.\ ,\\
]\p_0,\p_1[&\defeq&\{\p\mid p_0< p< p_1\}\ .	
\end{eqnarray}

We can observe the following fact:

\begin{proposition}\label{pro:edg}
	Given  two distinct paths $\p,\mk q$  so that  $]\p,\mk q[$ is empty, then there exists a complementary region so that $\{\p,\mk q\}=\{\partial^Le,\partial^Re\}$. In particular $\p,\mk q$ are rational.
\end{proposition}

\section{Gap function, error terms and the Gap inequality}\label{sec:gap}

We now assume that we have a positive harmonic 1-form $\Phi$ defined on our tree. Given a complementary region $e$ we define its {\em gap} as
$$
\G_\Phi(e)=\frac{1}{2}\left(\mu_\phi\partial^Le)+\mu_\phi(\partial^Re)\right)
$$
Our first result is the following:
 \begin{theorem}{\sc [Gap inequality]}
We have
\begin{eqnarray}
\sum_{e\in \mathcal F} \G_\Phi(e)&\leq&\frac{1}{2}\partial\Phi\ .
\end{eqnarray}
\end{theorem}
We will refer to
\begin{eqnarray}
\E(\Phi)\defeq
\frac{1}{2}\partial\Phi-\sum_{e\in \mathcal F} \G_\Phi(e)\ 
\end{eqnarray}
as the error term of $\Phi$.
This gap inequality is almost a tautology and we shall give two immediate proofs. One emphasizes the probabilistic nature of the situation, the other uses the Green Formula and is very close to Bowditch's original idea.

\subsection{Gap and error  terms using the hamonic measure on the space of paths}
We can express the gap and error terms using the harmonic measure
\begin{proposition}
	We have the equalities
	$$
	\G(\Phi)=\frac{1}{2}\mu_\Phi(\mc Q), \ \ \E(\Phi)=\frac{1}{2}\mu_\Phi(\mc P\setminus\mc Q)\ .
	$$
\end{proposition}
\begin{proof}
	 Since $\mu_\Phi(\mc P)=\partial\Phi$, the second equality is a consequence of the first. To prove the first we have to notice that any rational path appears  in the boundary of exactly one complementary region,
	 Thus
	 $$
	 \G(\Phi)=\sum_{e\in\mc F}\G_\Phi(e)=\frac{1}{2}\sum_{\p\in\mc Q} \Phi_\infty(\p)=\frac{1}{2}\mu_\Phi(\mc Q)\ .
	 $$
\end{proof}

\subsection{The gap inequality from the Green Formula}

For any complementary region  $f$ 
and $n\in\mathbb N$ let
$$
\G_\Phi^n(f)\defeq \frac{1}{2} \left(\Phi(\L^n\cdotp\R\cdotp f)+\Phi(\R^n\cdotp\L\cdotp f)\right).
$$
Then obviously, from the positivity of $\Phi$, we have
$$
\G_\phi^n(f)\geq \G_\Phi(f)\defeq\liminf_{n\to\infty}\G_\Phi^n(f).
$$
We can now revisit the Green Formula by rearranging its term to get

\begin{proposition}
We have
\begin{eqnarray}
\sum_{f\in \mathcal S(p)\mid p< n} \G_\Phi^{n-p}(e)&=&\frac{1}{2}\partial\Phi\ ,\label{eq:gapeq}
\end{eqnarray}
\end{proposition}
The gap inequality follows from this formula since $\G_\Phi(e)\leq \G^n_\phi(e)$.

\begin{proof} For every $n$, we have  the Green formula
\begin{align}
  \sum_{x\in \mathcal S(n)} \Phi(x)= \partial\Phi.
\end{align}
For all $n$, recall that
\begin{align}
 \G_\Phi^n(e)=\frac{1}{2}\left(\Phi(\L^n\cdotp\R\cdotp e)+\Phi(\R^n\cdotp\L\cdotp e)\right)\ .
\end{align}
Recall also that given an element $z$ in $\mathcal S(n+1)$ there exist unique element $z$ in $\mathcal S(p)$  with $p<n$  with 
$z\in \partial y$. or in other words $z=\L^{n-p}\cdotp\R\cdotp y$ or $z=\R^{n-p}\cdotp\L\cdotp y$.
Then we have
\begin{align}
  \sum_{x\in\mathcal S(p)\mid p<n} \G_\Phi^{n-p}(x)=& \frac{1}{2} \sum_{x\in\mathcal S(p)\mid p<n}  \left(\Phi(\L^{n-p}\cdotp\R\cdotp x)+\Phi(\R^{n-p}\cdotp\L\cdotp x)\right)\cr
   &=  \demi\sum_{y\in \mathcal S(n+1)} \Phi(y)=\demi \partial\Phi.
  \end{align}

\end{proof}

\section{The error term and random variables on $\mc P$}\label{sec:err}

The terminology "rational" and "irrational" paths seems to suggest that the measure of rational paths ought to be zero. This is indeed the case when the probability is balanced: we have equal probability to take the left or right edge. However, we shall see on the contrary that in the geometric context which is underlying McShane's identity the measure of irrational paths is zero.

We develop in this section a framework to understand the error term as the measure on some set on the circle. More precisely, we describe an  "increasing embedding" of $\mc P$ in $\mathbb R/\partial\Phi\mathbb Z$. This will be closer to the original point of view of Mc Shane's and will also us to define the error term as the Lebesgue measure of some set in the circle, which we call the Birman--Series phenomenon in Theorem \ref{cor:ErrXinf}.

\subsection{Gap and  error terms using random variables on  $\mc P$}
\label{sec:rv}
The Green Formula
$$
\sum_{e\in S(n)}\Phi(e)=\partial\Phi,
$$
as well as the circular order on $S(n)$
can be used to define a partition of $\mathbb R/\partial\Phi\mathbb Z$ in intervals labelled successively by edges in $S(n)$ and of length $\Phi(e)$. Such a partition is unique up to translation.

Using the extremities of this evolving partition rather than the length of the corresponding intervals, we express in Theorem \ref{cor:ErrXinf} the values of the error term as the Lebesgue measure of some set $\overline X_\infty$.

In our future geometric application, the vanishing of the Lebesgue measure of this set -- and consequently the vanishing of the error  term -- will be an  application of the Coding Theorem \ref{theo:exten} and   the Birman--Series Theorem \cite{Birman:1985} which is a cornerstone of the proof of McShane's type identity as in \cite{McShane:1998} and \cite{Labourie:2009wi}.

Taking the mid points of each such interval gives rise to an increasing  random variable $Y_n$ on $S(n)$, which is unique up to translation when characterised by
$$ 
Y_n(\p_0)-Y_n(\p_1)=\frac{1}{2}(\phi(\p^n_0)+\phi(\p^n_1)),
$$ when  $\p_0<\p_1$ and there is no edge between $\p^n_0$ and $\p^n_1$.

We may consider $Y_n$ as a random variable on $\mc P$, and 
increasing $n$, obtain a variable $X_\infty$. Finally the goal of this section is to define the error term in terms of $X_\infty$.

\subsubsection{The variables $Y_n$}

Let us choose one ``initial'' complementary region $f_0$ whose boundary contains the root. Then the fact that the tree is planar gives a natural ordering of the edges of $S(n)$.

For every $n$, we construct a function $Y_n$ on $S(n)$  with values in $\mathbb R/\partial\Phi\mathbb Z$, by
\begin{eqnarray*}
        Y_n(e)&\defeq& \frac{1}{2}\Phi(e)+\sum_{f\in S(n)\mid f<e}\Phi(f)\\ 
        &=&-\frac{1}{2} \Phi(e)+\sum_{f\in S(n)\mid f\leq e}\Phi(f).
\end{eqnarray*}

The following proposition is an immediate consequence of the definition of $Y_n$ and the positivity of $\Phi$.

\begin{proposition}\label{pro:yn}   For $e$ and $f$ in $S(n)$ with $e>f$ , we have the inequalities
\begin{eqnarray}
	Y_{n+1}(\R\cdotp f)  \leq Y_{n+1}(\L\cdotp e)&\leq& Y_n(e)\leq Y_{n+1}(\R\cdotp e)\ , \\
		\left\vert Y_{n+1}(g)-Y_{n}(e))\right\vert&\leq& \frac{1}{2}\Phi(h)\, \ \hbox{ for } \{g,h\}=\{\R(e),\L(e)\}\ ,\label{eq:Yrg}
	\end{eqnarray}
\end{proposition}

\subsubsection{The variables $X_n$ and $X_\infty$}
We then define for every path $\p$ in the set $\mc P$ of infinite embedded paths starting from the root.
\begin{align}
  X_n(\p)\defeq Y_n(\p^n)).
\end{align}
Anticipating Proposition \ref{pro:xinfty}, we shall see that the variables $X_n$ converges uniformly to a variable $X_\infty$.

Let us first prove the following Proposition:

\begin{proposition}\label{pro:opedges}
Let $\p$ be a path. Let $\hat \p=\{\hat \p^n\}_{n\geq 2}$ be the sequence of edges so that $\hat \p_n\in S(n)$ and $\{\p^n, \hat \p^n\}=\{\R(\p^{n-1}),\L(\p^{n-1})\}$.
Then
$$
\sum_{n=1}^\infty\Phi(\hat \p^n)<\infty\ ,
$$
and in particular
$$
\lim_{n\to\infty}\Phi(\hat \p^n)=0\ .
$$
\end{proposition}
\begin{proof}
	Since $\Phi(\p^p)+\Phi(\hat\p^p)=\Phi(\p^{p-1})$ it  follows by induction that for all $p$,
$$
\Phi(\p^p)+\sum_{n=2}^p\Phi(\hat \p^n)=\Phi(\p^{p-1})+\sum_{n=2}^{p-1}\Phi(\hat \p^n)=\Phi(\p_1 )\ .
$$
The result follows since $\Phi$ is positive.
\end{proof}

The main observation of this section is 
\begin{proposition}\label{pro:xinfty} We have 
\begin{enumerate}
\item The random variable $X_n$ is increasing with respect to the lexicographic ordering on $\mc P$.
	\item The sequence $X_n$ converges pointwise to a random variable  $X_\infty$. Moreover
\begin{eqnarray}
X_\infty(\p_1)-X_\infty(\p_0)&=&\mu_\Phi\left([\p_0,\p_1]\right)-
\frac{1}{2}\left(\Phi_\infty(\p_0)+\Phi_\infty(\p_1)\right)\ ,\label{eq:Xmu1}\\
&=&\mu_\Phi\left(]\p_0,\p_1[\right)+\frac{1}{2}\left(\Phi_\infty(\p_0)+\Phi_\infty(\p_1)\right)\ \label{eq:Xmu2}.
\end{eqnarray}
\end{enumerate}
\end{proposition}

\begin{proof} The first item is a consequence of Proposition \ref{pro:yn}
Notice now, using the notation of Proposition \ref{pro:opedges}, that  Equation  \eqref{eq:Yrg} implies  that 
\begin{equation}
\left\vert X_{n+1}(\p)-X_n(\p)\right\vert = \Phi(\hat\p^n)\ .
\end{equation}
Thus by Proposition \ref{pro:opedges}, $X_n$ converges pointwise. Finally, let $\p_0$ and $\p_1$ be two paths with $p_0<p_1$ in the lexigraphic order.
Let $[\p_0,\p_1]\defeq\{\p\mid p_0\leq p\leq p_1\}$. Then for $n$ large enough

\begin{eqnarray}
X_n(\p_1)-X_n(\p_0)
=\sum_{\stackrel{e\in S(n)}{ \p^n_0
\leq e\leq \p^n_1}}\Phi(e)-\frac{1}{2}\left(\Phi(\p^n_0)+\Phi(\p^n_1)\right).
\end{eqnarray}
Thus
$$
X_\infty(\p_1)-X_\infty(\p_0)=\mu_\Phi[\p_0,\p_1]-\frac{1}{2}\left(\Phi_\infty(\p_0)+\Phi_\infty(\p_1)\right).
$$
This proves equation \eqref{eq:Xmu1}, equation \eqref{eq:Xmu2} follows then from equation \eqref{eq:atom}.
\end{proof}

\subsubsection{The error term using $X_\infty$}
The set $X_\infty(\mc P)$ might be not be closed, for instance when $X_\infty$ is not continuous. Let us then define for every $\p\in \mc P$, 
\begin{eqnarray*}
X_\infty^L(\p)&\defeq&\sup_{\mathfrak q<\p}(X_\infty(\mk q))\ ,\\
X_\infty^R(\p)&\defeq&\inf_{\mathfrak q>\p}(X_\infty(\mk q))\ ,\\
\overline X_\infty&\defeq&\overline{\bigcup _{\p\in\mc P\setminus\mc Q}	[X^R_\infty(\p)),X_\infty^L(\p)}]\label{eq:gapX1}\ .
\end{eqnarray*}
Then we can express the error term using the variable $X_\infty$:
\begin{theorem}
{\sc [Birman--Series phenomenon]} \label{cor:ErrXinf}
	\begin{eqnarray}
	\lambda\left(\overline X_\infty\right)&=&2\E(\Phi).
	\end{eqnarray}	
	
\end{theorem}
This will be a consequence of the following proposition:

\begin{proposition}\label{pro:Xinfty}
Suppose that $\p$ is a limit of a strictly increasing sequence $\{\p_n\}_{n\in\mathbb N}$, and $\mk q$ is limit of a strictly decreasing sequence $\{\mk q_n\}_{n\in\mathbb N}$. Then
	\begin{eqnarray}
	X_\infty(\p)-X^-_\infty(\p)&=&\frac{1}{2}\Phi_\infty(\p)\,  \label{eq:gapX2}\\
		X^+_\infty(\mk q)-X_\infty(\mk q)&=&\frac{1}{2}\Phi_\infty(\mk q)\,  \label{eq:gapX2+}\\
     \left(\mathbb R/\partial\Phi\mathbb Z\right)\setminus \overline X_\infty&=&\bigsqcup_{e\in\mc F}]X^L_\infty(\partial^Le),X^R_\infty(\partial^Re)[\ \label{eq:gapX3}.
	\end{eqnarray}
	\end{proposition}

	We first prove Theorem \ref{cor:ErrXinf} from Proposition \ref{pro:Xinfty}.
\begin{proof}	
Since $]\partial^-e,\partial^+e[=\emptyset$, we have that from the previous proposition and equation \eqref{eq:Xmu2} ,
\begin{eqnarray*}
	X_\infty(\partial^Re)-X_\infty(\partial^Le)=\frac{1}{2}\left(\Phi_\infty(\partial^Le)+\Phi_\infty(\partial^Re)\right)=\G_\Phi(e).
	\end{eqnarray*}
	Moreover from Equations \eqref{eq:gapX2} and \eqref{eq:gapX2+}, for $i=R$ or $i=L$
	$$
	\vert X^i(\partial^ie)-X_\infty(\partial^i e)\vert=\frac{1}{2}\Phi_\infty(\partial^i e)\ .
	$$
	Thus combining the two preceding equations one gets that 
	$$
	\lambda(]	X_\infty^L(\partial^Le),X_\infty^R(\partial^Re)[)=	X_\infty^R(\partial^Re)-X_\infty^L(\partial^Le)=2\G_\Phi(e).
	$$
	
	Thus from equation \eqref{eq:gapX3}, we get 
	\begin{eqnarray*}
		\partial\Phi&=&\lambda(\overline X_\infty)+ \sum_{e\in\mc F}\lambda(]X_\infty^L(\partial^Le),X_\infty^R(\partial^Re)[)\\
		&=&	\lambda(\overline X_\infty)+ 2\sum_{e\in\mc F}\G_\Phi(e)\ .
	\end{eqnarray*}
	Hence, 	$\lambda(\bar X_\infty(\mc P))=\partial\Phi- 2\G(\Phi)=2\E(\Phi)$. \end{proof}

	Let us now prove  Proposition \ref{pro:Xinfty}
	\begin{proof}  Given a path $\p$ and an infinite strictly  incresasing sequence  $\{\p_n\}_{n\in\mathbb N}$  converging to $\p$, then  $\{X_\infty (\p_n)\}_{n\in\mathbb N}$ converges to $X^-_\infty(\p)$. In particular, $\{X_\infty(\p_n)-X_\infty(\p_{n+1})\}_{n\in\mathbb N}$ converges to zero. Since  by Equation \eqref{eq:Xmu2},
	\begin{eqnarray*}
   X_\infty(\p_{n+1})-X_\infty(\p_n)=\mu_\Phi(]\p_n,\p_{n+1}[)+\frac{1}{2}(\Phi(\p_n)+\Phi(\p_{n+1}))\ ,
\end{eqnarray*} 
	the sequence $\{\Phi_\infty(\p_n)\}_{n\in\mathbb N}$ converges to zero.
Then by Equation \eqref{eq:Xmu2} again,
\begin{eqnarray*}
   X_\infty(\p)-X_\infty(\p_n)=\mu_\Phi(]\p_n,\p[)+\frac{1}{2}(\Phi(\p_n)+\Phi(\p))\ . 
\end{eqnarray*}
Thus taking the limit, one gets Equation \eqref{eq:gapX2}. A symmetric argument yields Equation \eqref{eq:gapX2+}.

For the last statement, let $u\not\in \overline X_\infty$.  Since $\overline X_\infty$ is closed, let $u\in ]v,w[\subset \mathbb R/\partial\Phi\cdotp\mathbb Z\setminus \overline X_\infty$ and let
$$
\mc P^L=\{\p\mid X^L_\infty(\p)\leq v\},\ \ \mc P^R=\{\p\mid X^R_\infty(\p)\geq w\}.
$$
Let now $\p^R=\inf(\p\in\mc P^R)$ and  $\p^L=\sup(\p\in\mc P^L)$. By definition of $u$ it follows that $X^L_\infty(\p^L)\leq v$ and symmetrically that $X^R_\infty(\p^R)\leq w$.  Let $\q\in]\p^L,\p^R[$. Then $X^L_\infty(\q)\geq w >u$ and   $X^R_\infty(\q)\leq v <u$. This is impossible, hence  $]\p^L,\p^R[=\emptyset$. It follows that $\p^i =\partial^i e$ for some complementary region $e$ and $u\in ]X^L_\infty (\partial^Le),X^R_\infty (\partial^Re)[$. Conversely, if $u\in ]X^L_\infty (\partial^Le),X^R_\infty (\partial^Re)[$ then $u\not\in \overline X_\infty$.
\end{proof}

\section{A rooted tree coding geodesics}\label{sec:tree}

This section is purely topological.  Let $S$ be a compact surface with one marked point $x_0$ and $\Sigma=S\setminus\{x_0\}$. Let $N=6(1-\chi(S))$.  Let $\mathcal T$ be the rooted trivalent tree, which has $N$ edges at the root. As a corollary of our main result we will obtain the following result.

\begin{theorem}
Let $T$ be a triangulation of $S$ whose only vertex is $x_0$. There there exists a labelling of the edges of $\mc T$ by simple curves passing though $x_0$. 
\end{theorem}
Our actual results, Theorem \ref{theo:topcod} and \ref{theo:exten} will be more precise and will allow us to describe simple infinite geodesics starting at $x_0$ as embedded paths in $\mc T$.

The main proposition  is actually an elaboration on a Lemma by L. Mosher \cite{Mosher:1995dw} as Saul Schleimer has explained to us and  the proof  is actually identical. Our point of view and purpose are nevertheless different.

\subsection{A dynamical system on the set of triangulations}

Let $S$ be a closed oriented surface with one puncture $x_0$ and $\Sigma=S\setminus\{x_0\}$. An {\em ideal triangulation} of $\Sigma$ 
is a triangulation of $S$, up to isotopy, whose only vertex is $x_0$.  Observe that the number $N$ of oriented edges is $6(1-\chi(S))$: the singular flat metric obtained by identifying the triangles with equilateral triangles of length 1, has total curvature $N\pi/3-2\pi$. Since this total curvature is $-2\pi\chi(S)$, we obtain that $N=6(1-\chi(S))$.

We will usually think of this triangulation as a triangulation by ideal triangles of the surface $\Sigma$ equipped with a complete hyperbolic structure. Given a triangulation $T$ and an oriented edge $e$, let $\bar e$ be the same edge with the opposite orientation, and $s(e)$  the oriented edge whose origin is the end point of $e$ such that $e$ and $s(e)$ are both on the boundary of the triangle on the right of $e$.

Let $\mc S$ be the space of pairs $(T,e)$ where $T$ is an ideal triangulation of $S$ and $e$ an oriented edge of $T$. We are going to describe three transformations $\F$, $\R$ and $\L$ on $\mc S$.

First $\F(T,e)=(\widehat T,\widehat e)$ where $\widehat T$ is the triangulation flipped at $e$, and $\widehat e$ the new corresponding edge so that $(e,\widehat e)$ is positively oriented.

Given $(T,e)$, define $\R(T,e):=\F(T,s(e))$. 
\begin{figure}[h]
  \centering
  \subfloat[$(T,e)$]{\label{fig:Te}\includegraphics[width=0.3\textwidth]{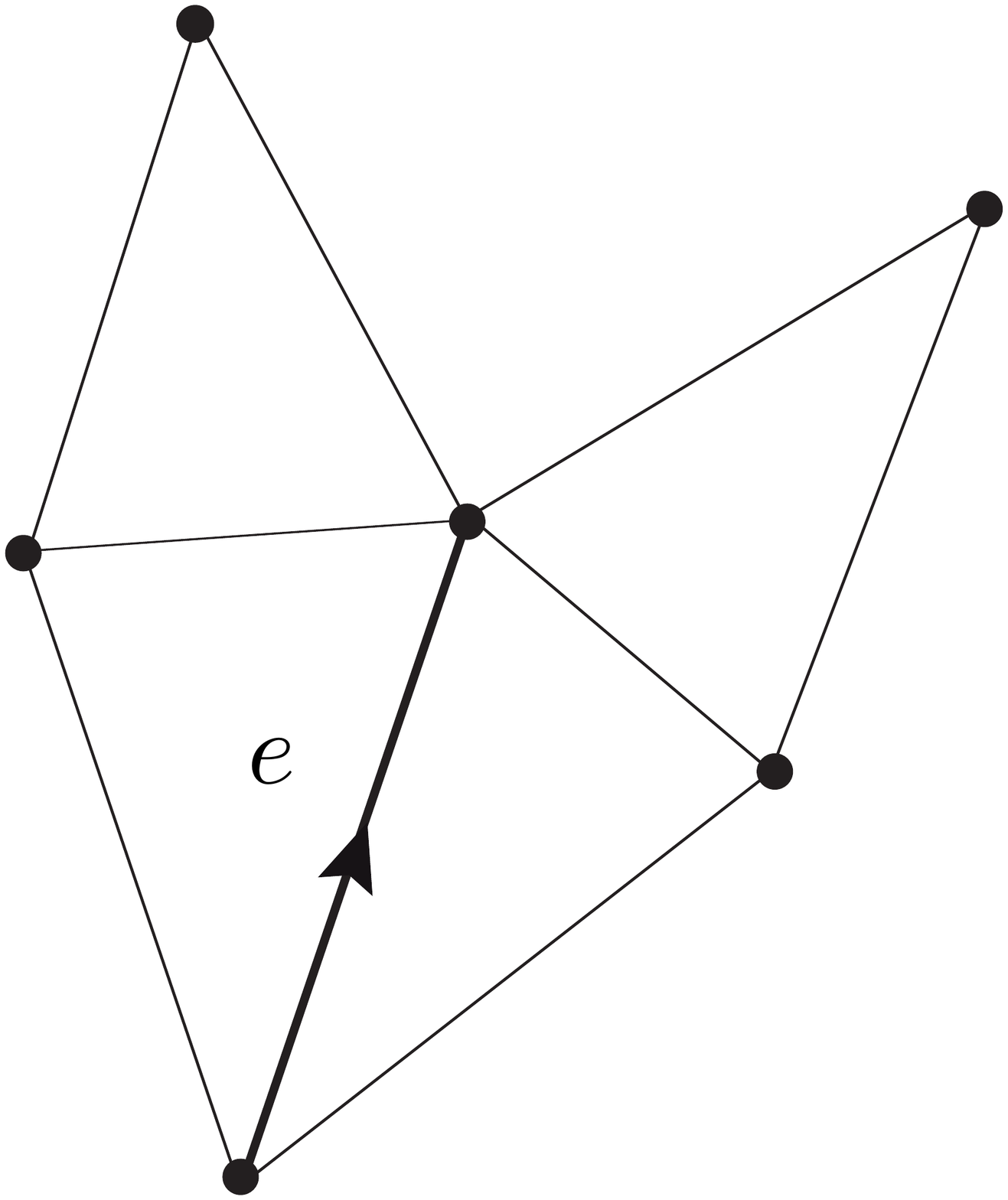} }               
  \subfloat[$\R(T,e)$]{\label{fig:Rte}\includegraphics[width=0.3\textwidth]{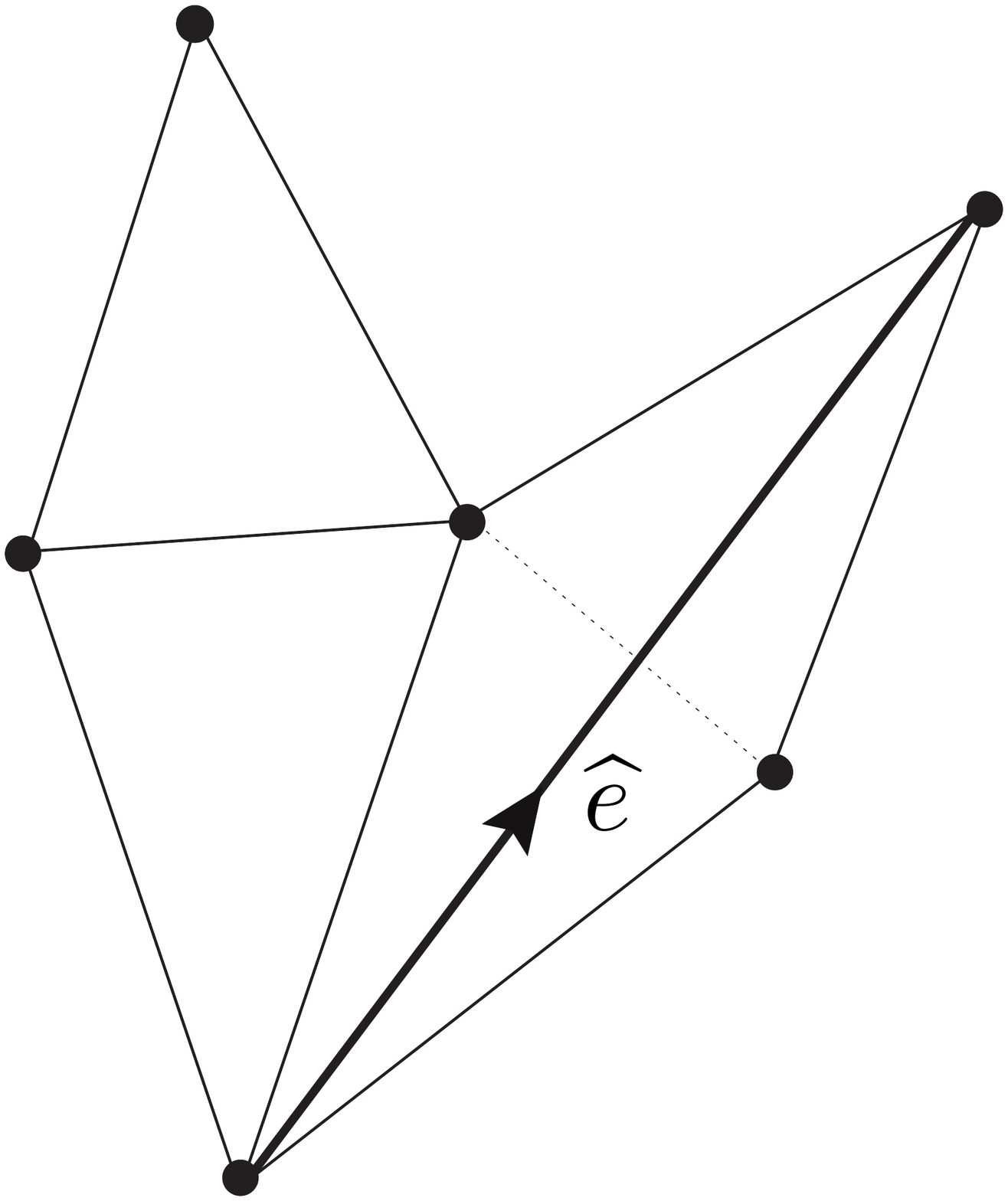}}
    \subfloat[$\L(T,e)$]{\label{fig:Lte}\includegraphics[width=0.3\textwidth]{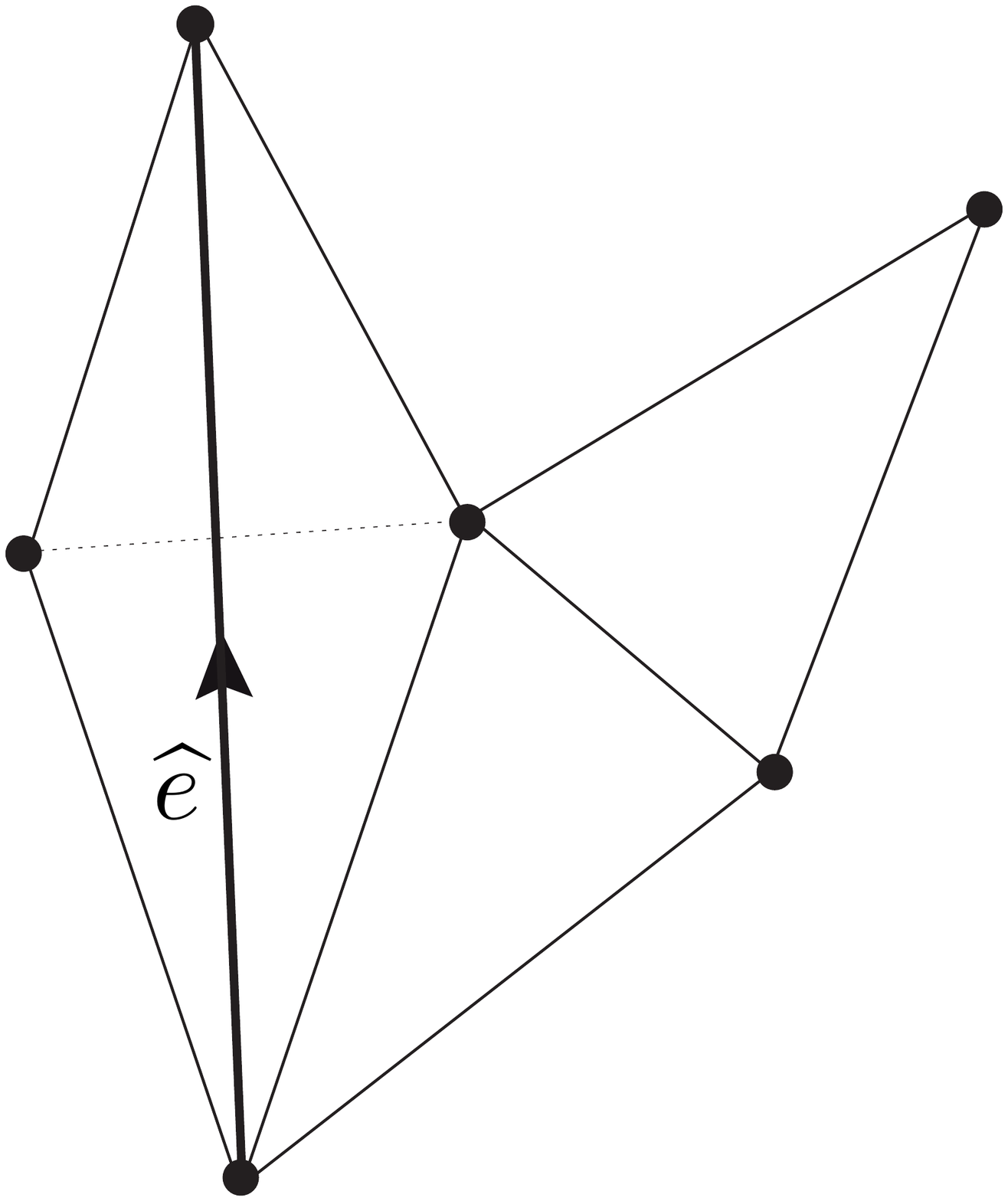}}
  \caption{Transformations $\R$ and $\L$}
  \label{fig:RL}
\end{figure}

A similar definition holds for $\L$ after changing the orientation of $\Sigma$. All together, these two transformations define an action of the free semigroup in two generators $\Fd$ on $\mc S$. 

Let also $L(x_0)$ be the space of simple loops, up to homotopy, passing through $x_0$ and $\pi$ the natural projection from $\mc S$ defined by $\pi(T,e)=e$. 

Our main result is the following

\begin{theorem}{\sc [Topological coding]}\label{theo:topcod}
Let $T_0$ be an ideal triangulation and $f$ an oriented topological embedded loop on $S$ through $x_0$. Then, either $f$ is an edge of $T_0$ or there exists a unique oriented edge $e$ of $T_0$ and a unique element $g$ in $\Fd$ such that
$\pi(g\cdotp\F(T_0,e))=f$. 
\end{theorem}
 
In particular, the action of $\Fd$ on $\mc S$ is free. We will explain later a more precise version of this result.

We will also explain that this construction give rise to a coding of all embedded geodesics without self intersection issuing from $x_0$ in an auxiliary complete hyperbolic metric on $\Sigma$. A further property of this coding yields the Birman--Series Theorem.

We will also explain using this coding that every Hölder cross ratio on $\partial_\infty$ gives by a harmonic measure construction a transverse measure with zero entropy on the set of simple infinite geodesics (emanating from a boundary component).

Then, generalising an idea from Bowditch in the punctured torus case \cite{Bowditch:1996}, this construction will yield a new proof --or rather a new interpretation--  of McShane identity for cross ratios \cite{Labourie:2009wi}.

\subsection{An arboreal interpretation}

For later use and in order to make the connection with the previous constructions, we present the results in terms of a planar tree.

Let $T$ be a triangulation with $n$ edges whose only vertex is $x_0$. then let $\mathcal T$ be the rooted trivalent (except at the root) planar tree with $2n$ root at the origin. 

We first define a labelling of the edges and the complementary regions as follows:
\begin{itemize}
	\item We label the initial edges by the  pairs $\F(T,e)$ using the cyclic order that comes from the cyclic order on the oriented  simple curves $\pi(\F(T,e))$.
	\item Next, label every other edge recursively so that if an edge $a$ of the tree is labelled by $(T',e')$, the edge on the right of $a$ is labelled by $\R((T',e'))$ and the edge on the left is  labelled $\L((T',e'))$.
	\item We  label the  interior-most complementary regions $f$ (those adjacent to the root $v_0$) by $(T,e)$, where $\R(T,e)$ is the label of the first edge of the tree on the right of $f$. It then follows that $\L(T,e)$ is the label of  the first edge of the tree on the left of $f$.
	\item Finally, we label finally the other complementary regions by the label of the edge of the tree defining it.
\end{itemize}
Our theorem now can be rewritten as

\begin{theorem}{\sc [Labelling complementary regions]}\label{theo:topcod2} Let $T$ be a triangulation of $S$ with $n$ edges whose only vertex is $x_0$. Let $\Psi$ be the labelling of the complementary regions of 
 $\mc T_n$  by  pairs $(T',e')$ where $e'$ is an oriented edge of the triangulation $T'$ -- described above. Then the map 
	$f\mapsto \pi(\Psi(f)$ is a bijection from the space of complementary region to the space of simple curves passing through $x_0$.
\end{theorem}

\subsection{Order preserving} 
In this section, we equip $\Sigma$ with an auxiliary complete hyperbolic metric so that the triangles are realized as ideal hyperbolic triangles

Let $L(x_0)$ be the set of simple oriented geodesics based at $x_0$. Observe that the orientation of $\Sigma$, as well as that of an auxiliary hyperbolic metric, gives a cyclic ordering on $L(x_0)$ eventually independent of the metric: we order the geodesics $\gamma$  by the order on a small horosphere $H$ at $x_0$ of the first intersection point in $H\cap\gamma$.

We now  define a map $\Psi$  from $E\times\Fd$ to  $L(x_0)$ by
\begin{align}
 \Psi(e,g)\defeq \pi\left(g\cdotp\mk F(T_0,e)\right).
 \end{align}
 We consider the cyclic ordering on $E$ induced by $e\mapsto\Psi(e,1)$ from the order on $L(x_0)$.
%
%

We fix a triangulation $T$.  We  define a cyclic order (by lexicographic ordering) on $E\times\Fd$. We now prove
\begin{proposition}\label{orderP}
The map $\Psi$ from $E\times \Fd$ to $L(x_0)$ 
is order preserving.
\end{proposition}

As an immediate corollary, we get that the action of $\Fd$ on $\mc S$ is free.

\begin{proof}
We will prove the proposition by induction on the length of the word in $\Fd$. Since $E$ has $N$-elements, the simple loops $E$ and the simple loops of the form $\Psi(e,1)$  divides the horocycle into $2N$ intervals so that elements of $E$ and elements of 
$\Psi(E,1)$ alternate. Furthermore, by the definition of the action of $\Fd$, we see that when $g$ has word length one, $\pi(g\cdotp\mk F(T_0,e))$ lies in  the interval to the left or right of the corresponding element $e \in E_{x_0}$ depending whether $g=\R$ or $g=\L$. The new $2N$ edges that are introduced alternate with the original $2N$ edges with respect to the cyclic ordering and again, by construction every pair of alternating edges (one of these being $\pi(g\cdotp\mk F(T_0,e))$) bound a triangle of the triangulation associated to $g\cdotp\mk F(T_0,e)$. Continuing inductively, we see that at any distance $d>1$ from the root, we have $2^{d}N$ edges whose cyclic ordering agrees with that of $E\times \Fd$ which completes the proof.\end{proof}

\subsection{Reducing the complexity}

We are going to define in this paragraph a {\em complexity invariant} associated to a triple $(T,\Delta,\eta)$ where
\begin{itemize}
\item $T$ is a triangulation by geodesics arcs for a complete (auxiliary) hyperbolic metric of $\Sigma$,
\item $\eta$ is an oriented  simple geodesic starting at the puncture $x_0$,
\item $\Delta$ is a {\em fundamental domain} for the surface with respect with the triangulation, that is a map  from a triangulated disk $D$ to $S$, which sends triangles to triangles, so that $\Delta$ restricted to the interior of $D$ is injective with a dense image.
\end{itemize}

Our {\em complexity invariant} is a pair of integers
\begin{align}
C(T,\Delta,\eta)= (n(T,\Delta,\eta),N(T,\Delta,\eta)),
\end{align}
defined in the following way. First if $\eta$ is an edge of $T$, then 
\begin{align}
C(T,\Delta,\eta)= (1,1).
\end{align}
If $\eta$ is not an edge, let $(\eta_1,\eta_2,\ldots, \eta_N)$ be the (ordered) connected components of $\Delta^{-1}(\eta)$. Then
\begin{enumerate}
\item $
N(T,\Delta,\eta)$ is the number of connected components of $\Delta^{-1}(\eta)$, that is $
N(T,\Delta,\eta)=N$,
\item $n(T,\Delta,\eta)$ is the number of triangles of $D$ that $\eta_1$ intersects.
\end{enumerate}
Our two main results are the following
\begin{proposition}\label{reduccomp1}
Assume that $n(T,\Delta,\eta)=1$ and $
N(T,\Delta,\eta)>2$, then there exists $\mathring \Delta$ so that
\begin{align}
 N(T,\mathring \Delta,\eta)&\leq N(T,\Delta,\eta)-1,\\
 2&\leq n(T,\mathring \Delta,\eta).
\end{align}
\end{proposition}
The second proposition is as follows.
\begin{proposition}\label{reduccomp2}
Assume that $n(T,\Delta,\eta)>1$ and let $e$ be the first edge of $T$ that $\eta$ meets. Let $(\mathring T,\mathring e)\defeq \F(T,e)$. Then $\Delta$ is still a fundamental domain for $\mathring T$ and furthermore,
\begin{align}
 N(\mathring T,\Delta,\eta)&=N(T,\Delta,\eta),\cr
  n(\mathring T,\Delta,\eta)&=n(T,\Delta,\eta)-1.
\end{align}
\end{proposition}

\subsection{Proof of proposition \ref{reduccomp1}}

By assumption there is a (closed) triangle $\delta$ in $D$ such that $\eta_1\subset \delta$. Let $e_0$ be the edge of $\delta$ opposite to the origin of $\eta_1$, $e_+$ and $e_-$ be the edges of $\delta$ on the left, respectively right of $\eta_1$. Let also $D^+$ be the connected component of $D\setminus\delta$ whose closure contains $e^+$ and let us define similarly $D^-$ (See figure \ref{fig:Delta}).
\begin{figure}[h]
\includegraphics[width=0.6\textwidth]{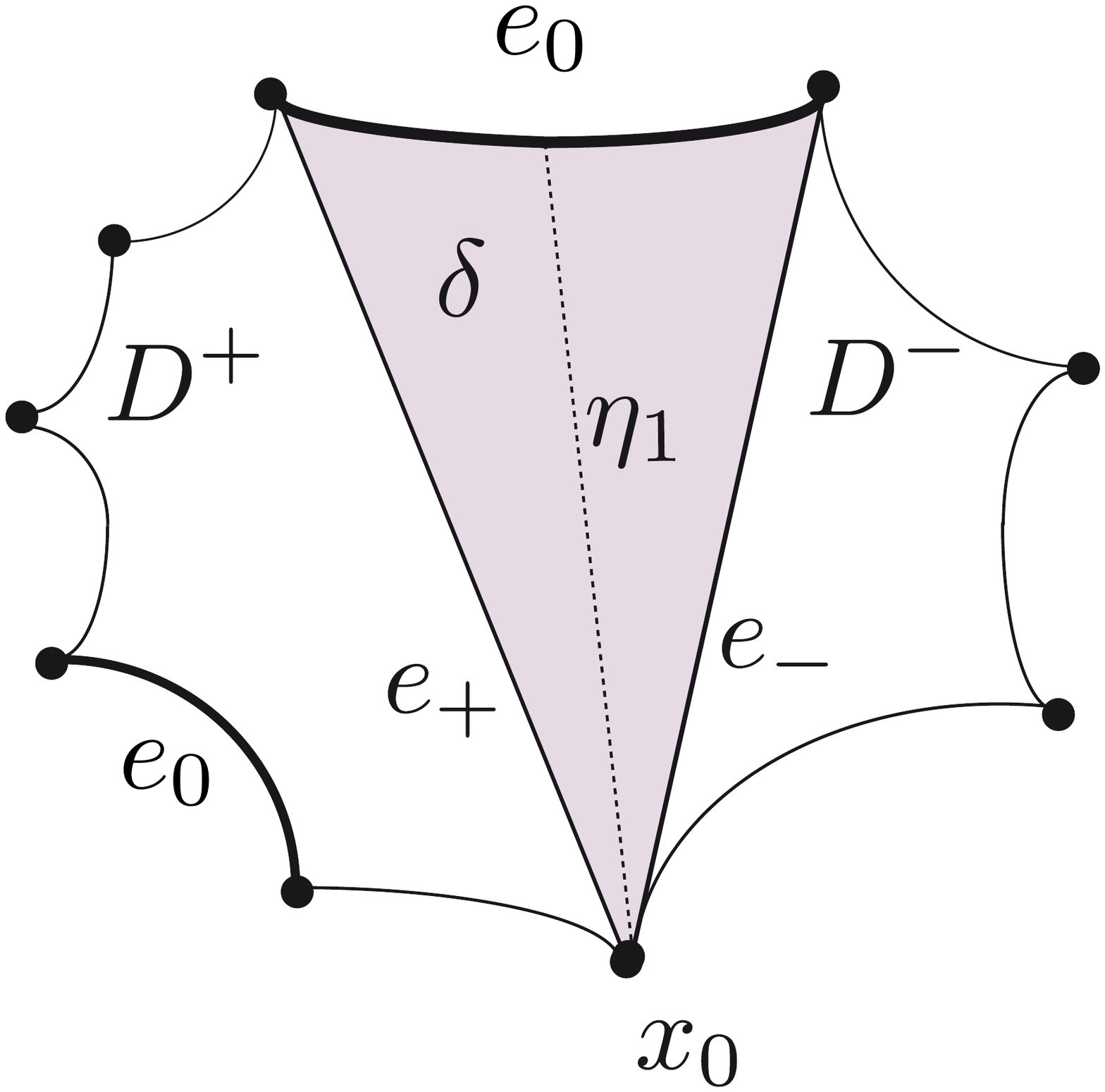}               
\caption{The subdomains of $\Delta$} \label{fig:Delta}
\end{figure}

Now, we may as well assume that $e_0$ appears on the boundary of $D^+$, the case where $e_0$ appears on the boundary of $D^-$ is handled symmetrically. We remark that $D^-$ could be empty.

Let then 
\begin{align}
  D^-_0\defeq D^-\cup \delta.
\end{align}
We now define a new triangulated disk and a new fundamental domain by
\begin{align}
  \mathring D\defeq D_0^-\cup_{e_0}D^+,
\end{align}
where the notation means that we glue $D_0^-$ and $D^+$ along $e_0$. The fundamental domain is then constructed accordingly. Observe that $D^+$ is non empty and thus
\begin{align}
  n(T,\mathring \Delta,\eta)\geq 2. 
\end{align}

\begin{figure}[h]
  \centering
  \subfloat[The old $\Delta$]{\label{fig:1a}\includegraphics[width=0.4\textwidth]{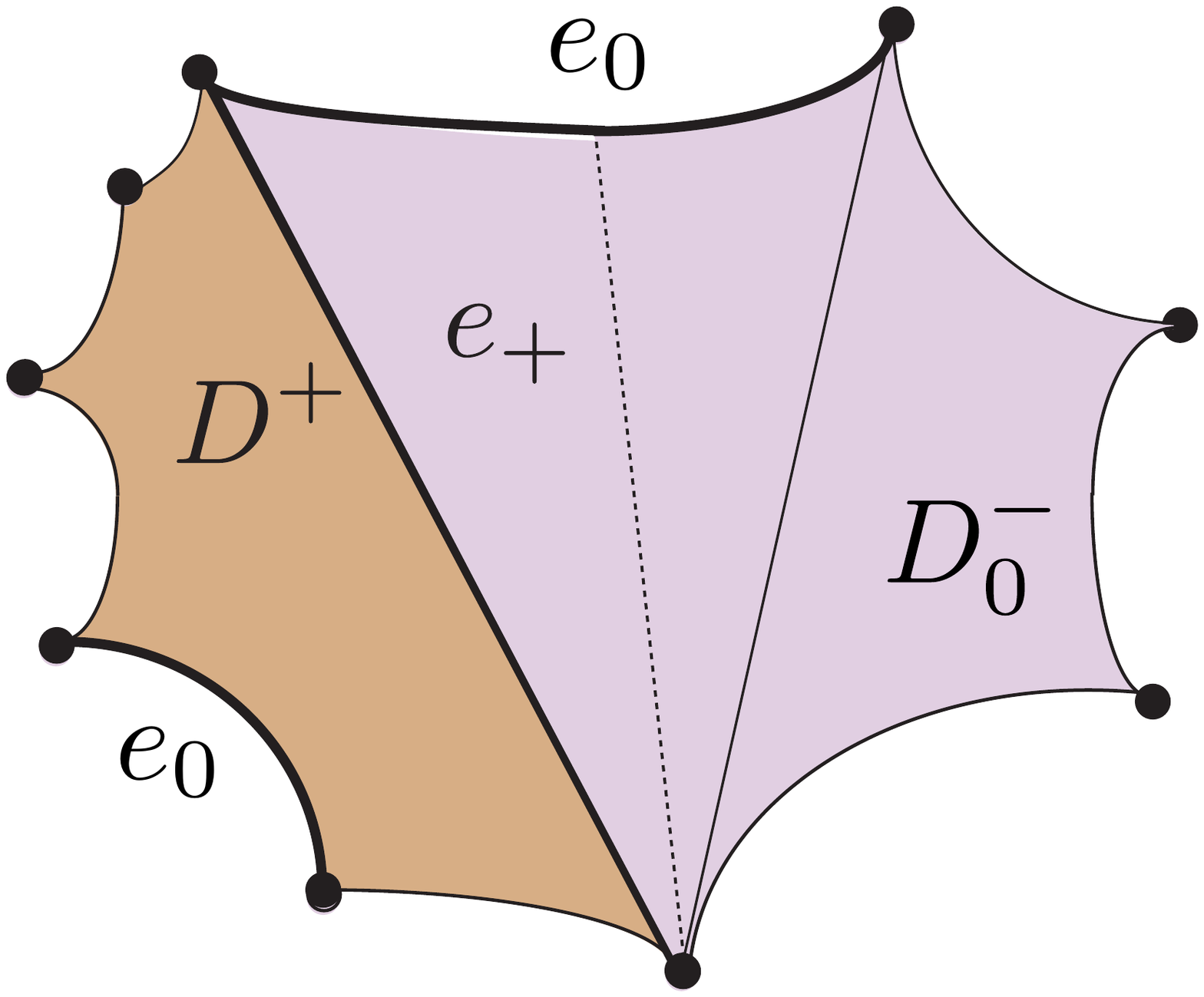} }               
  \subfloat[The new $\mathring \Delta$]{\label{fig:1b}\includegraphics[width=0.4\textwidth]{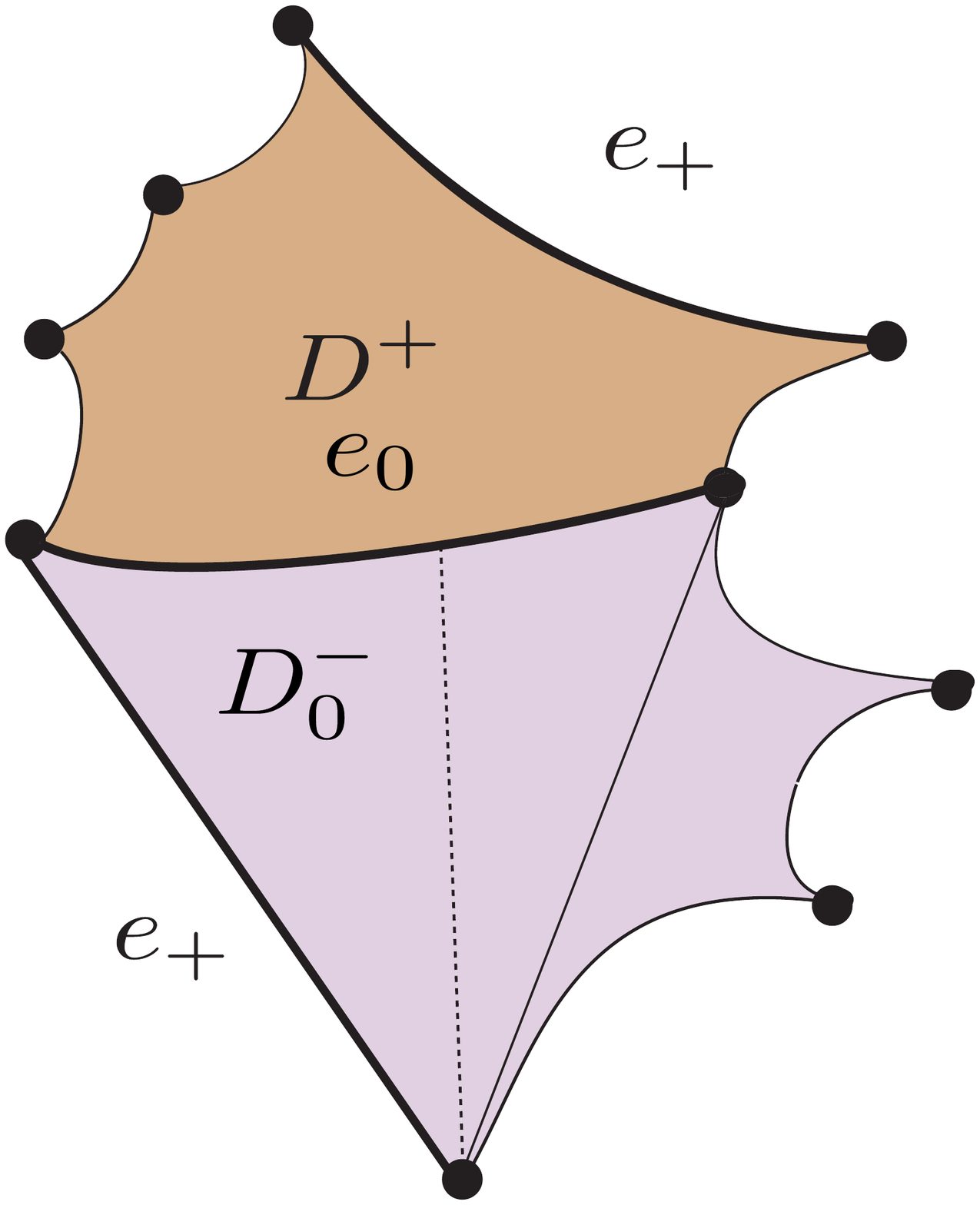}}
  \caption{Mutating the fundamental domain}
  \label{fig:1}
\end{figure}
Our proposition thus reduces to the following assertion
\begin{align}
 N(T,\mathring \Delta,\eta)\leq N(T,\Delta,\eta)-1, 
\end{align}
which we now prove.

Let $\phi$ be a collection of edges of $T$ and $i(\phi,\eta
)$ be the cardinal of the intersection of $\eta$ and $\phi$. To make things precise, $\phi$ is a collection of curves in $\Sigma$ (not in $D$) and thus there is no repetition of edges.

We now observe that if $\Delta$ is a fundamental domain and $\partial\Delta$ is the collection of edges of $T$ of the boundary of $\Delta$, then
$$
N(T,\Delta,\eta)=i(\partial\Delta,\eta)+1.
$$
Now by construction
$$
\partial\mathring \Delta=\left(\partial\Delta\setminus\{e_0\}\right)\cup \{e_+\}.
$$ 
Thus 
$$
N(T,\Delta,\eta)-N(T,\mathring \Delta,\eta)=i(e_0,\eta)-i(e_+,\eta).
$$
But since $\eta_1$ is embedded and goes from the vertex $x_0$ to $e_0$, every $\eta_i$ that intersects $e_+$ intersects $e_0$ as well (See figure \ref{fig:delta}). Furthermore, since $\eta_1$ intersects $e_0$ but not $e_-$ it follows that
\begin{align}
  i(e_0,\eta)-i(e_+,\eta)=1 +i(e_-,\eta)\geq 1.
\end{align}
This is what we wanted to prove.
\begin{figure}
  \centering
  \includegraphics[width=3in]{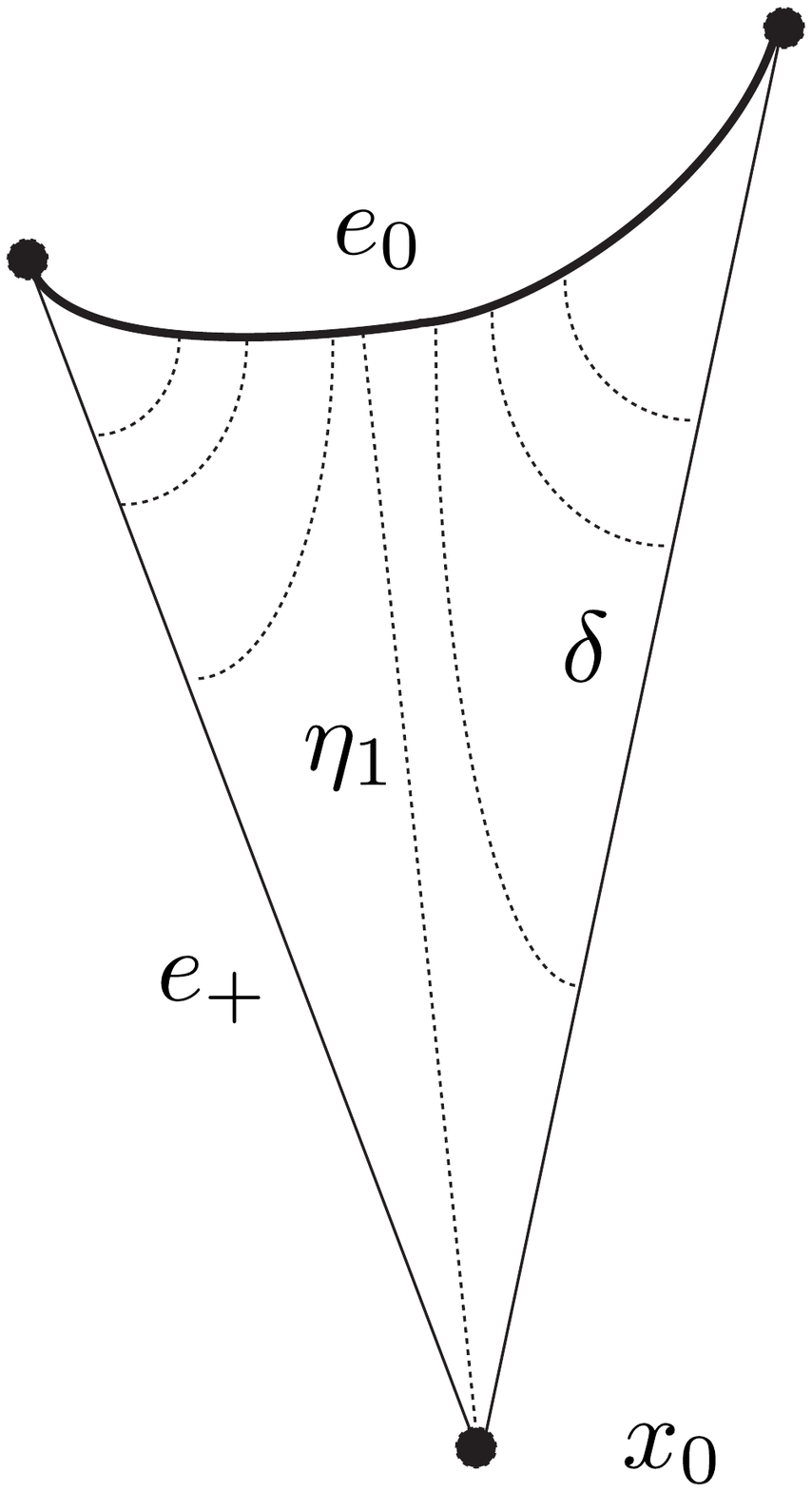}
  \caption{Intersection of $\eta$ with $e_0$ and $e_+$}
  \label{fig:delta}
\end{figure}
\subsection{Proof of proposition \ref{reduccomp2}}

By assumption, $n(T,\Delta,\eta)\geq 2$. Thus if $\delta_1$ and $\delta_2$ are the first triangles that $\eta_1$ encounters they are both in $D$. Thus, after flipping $\delta_1\cup\delta_2$ along their common edge, we see that $\Delta$ is still a fundamental domain for $\mathring T$ and thus obviously
\begin{align}
 N(\mathring T,\Delta,\eta)&=N(T,\Delta,\eta).
\end{align}
Also, in the new triangulation of $D$,  $\eta_1$ hits one less triangle, thus
\begin{align}
  n(\mathring T,\Delta,\eta)&=n(T,\Delta,\eta)-1,
\end{align}
This concludes the proof.

\subsection{The coding}

We now prove the coding theorem. The injectivity of the coding follows from Proposition \ref{orderP}. The sujectivity from the following

\begin{proposition}
There exists constant $A$ and $B$ depending only on the topology of $S$ so that, 
if $f$ be any simple loop based at $x_0$, if  $T_0$ is any triangulation, then there exists $e$ in $T_0$, an element $g$ of $\Fd$ such that
\begin{align}
  \pi(g\cdotp\F(T_0,e))=f,\label{def:coding}
\end{align}
and moreover, denoting by $\lambda$ the word length of $g$,  we have the rough inequality
\begin{align}
 \lambda\leq  A\cdotp i(f,T_0)+B
\end{align}
\end{proposition}

By definition, the  length $\lambda$ of the element $g$ satisfying Equation \eqref{def:coding} will be called the {\em tree length} of $f$ and denoted $\lambda_{T_0}(f)$ where the subscript is omitted when no confusion is possible, similarly $i(f,T_0)$ will be called the {\em intersection length} of $f$.

\begin{proof} We first want to prove the existence of $g$ satisfying Equation \eqref{def:coding}. 

Let $\eta$ be a simple geodesic starting at the puncture $x_0$. 
Let us choose a fundamental domain $\Delta$ for the initial triangulation $T$.

From the proof in the preceding paragraph,  one sees that 
\begin{align}
\lambda(f)\leq \sharp\{T_0\}\cdotp( N(T_0,\Delta,f)+1).
\end{align}

Similarly one observes that 
\begin{align}
  N(T_0,\Delta,f)-1=i(f,\partial\Delta)\leq i(f,T_0)\leq 2\cdotp\sharp\{T_0\}\cdotp N(T_0,\Delta,f).
\end{align}

The result follows. \end{proof}

\subsection{Spiraling}

In this section, we equip $\Sigma$ with an auxiliary complete hyperbolic metric of finite volume. Let  $L(x_0)$ be the set of oriented simple geodesic arcs from $x_0$ to $x_0$ and $\mc P$ be the set of embedded paths in the tree $\mc T$. Let also $\mc Q$ be the set of rational paths as defined in Paragraph \ref{sec:rat-path}. Given a complementary region $e$, recall that we have two rational paths $\partial^+e$ and $\partial^- e$.

Our goal is two prove that the simple arcs labelling  $\pi_n(\partial^\pm e)$ both converges to laminations spiralling around a closed geodesic in a precise sense. 

We first define the closed geodesic which the lamination spirals around. For $f \in L(x_0)$, define $f^R$ to be the oriented closed geodesic in $\Sigma$ homotopic to $f$ in $S$ on the right of $f$. Then $f$ and $f^R$ bounds a cylinder $Cyl^R(f)$ in $\Sigma$ with one cusped geodesic boundary corresponding to $f$ on the left, and one geodesic boundary corresponding to $f^R$ on the right. Let $f_{\infty}^R$ be the unique geodesic lamination in $Cyl^R(f)$ starting from $x_0$ and spiralling around $f^R$, respecting the orientation. 

We have a similar definition for $f^L$, $Cyl^L(f)$ and $f_{\infty}^L$, replacing right by left in the above. 

If now $F$ is a complementary region, by an abuse of notation, we denote also by $\pi_n(\partial^iF)$ the closed geodesic arc from $x_0$ to $x_0$ representing the simple arc labelling the edge $\pi_n(\delta^iF)$ for $i=L$ or $i=R$. Our result is now

\begin{proposition}\label{spiral} Let $F$ be a complementary region labelled with the simple arc $f$.
The sequence of  geodesic arcs   $\{\pi_n(\partial^RF)\}_{n\in\mathbb N}$  converges to $f_\infty^R$ and similarly    $\{\pi_n(\partial^RF)\}_{n\in\mathbb N}$  converges to $f_\infty^L$.
\end{proposition}

\begin{proof}
	The proof follows from (a) a simple observation  about geodesics in the cylinder $Cyl^R(f)$ originating from $x_0$, (b) Theorem \ref{theo:topcod2} on topological coding, and (c) proposition \ref{orderP} on order preserving. 
	
	First observe that every simple geodesic arc in $ L(x_0)$ whose beginning lies in $Cyl^R(f)$ must intersect $f^R$. In particular, the simple arcs $d_n:=\pi_n(\partial^RF)$ all intersect $f^R$ and $f_{\infty}^R$ is a left bound for $d_n$ with respect to the cyclic ordering. Next, the cyclic ordering implies that $d_n$ are monotone to the left, so $\lim_{n \rightarrow \infty}d_n$ exists.  Let $D$ be the positive Dehn twist about $f^R$, then $lim_{n\rightarrow \infty}D^n(d_0)=f_{\infty}^R$. If $\lim_{n \rightarrow \infty}d_n=l \neq f_{\infty}^R$, then there exists $m$ such that $D^m(d_0)$ is between $l$ and  $f_{\infty}^R$. But the order preserving property now implies that for $k$ sufficiently large, $d_k$ is between $f_{\infty}^R$ and $l$ which is a contradiction. Hence, $lim_{n\rightarrow \infty}d_n=f_{\infty}^R$.

\end{proof}

\subsection{Coding non self intersecting geodesics}
The previous proposition extends in the following way. Let $L_\infty(x_0)$ be the set of non properly embedded simple arcs from $x_0$. The proof of this Theorem follows immediately from original arguments by McShane's in \cite{McShane:1998}.

\begin{theorem}{\sc [Coding simple arcs and laminations]}\label{theo:exten}
The map $\Psi$ extends in a unique way to a bijection $\Psi_\infty$ from $\mc P$ to $L_\infty(x_0)$, so that if $\p\in\mc P$  then
$$
\lim_{n\to\infty}\Psi(\p^n)=\Psi_\infty (\p).
$$
\end{theorem}

\begin{proof} We just proved that $\Psi$ extend to $\Psi_\infty$  for rational paths, $\Psi_\infty$ -- restricted to irrational path -- is monotone. Let us now consider an irrational path $\p$. Let us define
\begin{eqnarray*}
	\Psi^-_\infty(\p)&=&\sup\{\Psi_\infty(\q)\mid \q\in \mc Q, \q<\p\}\ ,\\
	\Psi^+_\infty(\p)&=&\inf\{\Psi_\infty(\q)\mid \q\in \mc Q, \q>\p\}\ .
\end{eqnarray*}

If $w$ is a properly embedded simple geodesic arc from $x_0$ to $x_0$, then either $\partial^Rw<p$ or $\partial^Lw>p$ for the lexicographic order on paths coming from the coding. Since $\Psi_\infty(\partial^Lw)<\Psi(w)<\Psi_\infty(\partial^Rw)$, it follows that either $\Psi(w)<\Psi^-_\infty(\p)$ of $\Psi^+_\infty(\p)<\Psi(w)$. Now by McShane's gap Lemma \cite{McShane:1998},  if we have two distinct non properly embedded simple arcs starting at $x_0$, then there is always a simple properly embedded arc in the sector between  them. Thus  $\Psi^-_\infty(\p)=\Psi^+_\infty(\p)$. This concludes the proof. \end{proof}

\section{Cross ratio and a harmonic form}\label{sec:cr}

We now explain in this section how the geometry -- be it  hyperbolic, or that corresponding to a higher rank representation of the surface group -- has a counterpart in our planar tree picture as a harmonic 1-form in the sense of the first section.
\subsection{Cross ratio}
We now consider the boundary at infinity $\bgrf$  of $\grf$. Every element $\alpha$ in $\grf$ has two fixed points on $\bgrf$: the attracting fixed point  $\alpha^+$, and the repelling fixed point $\alpha^-$. 

We consider a Hölder equivariant cross ratio $\bb$ on $\bgrf$ (See \cite{Labourie:2009wi,Labourie:2005a,Labourie:2005} for definitions).

Recall the the length of an element $\gamma$ is then
\begin{align}
  \ell(\gamma)\defeq\log\left(\bb(\gamma^-,\gamma^+,y,\gamma.y\right),
\end{align}
where $y$ is any element of $\bgrf$ not fixed by $\gamma$.

A triangle $\Delta$ in a triangulation $T$, gives rise to three peripheral elements $\alpha$, $\beta$ and $\gamma$ of $\grf$, so that the triple $(\alpha,\beta,\gamma)$ is well defined up to gobal conjugation.

Similarly an edge $e$ of the triangulation $T$ gives rise to four peripheral elements that we denote by
\begin{align}
\delta(T,e),\ \alpha_1(T,e),\ \gamma(T,e),\ \alpha(T,e),
\end{align}

which corresponds to the vertices of the lozenge defined the the two triangles bounded by $e$, labelled using the cyclic ordering and starting at the initial vertex of $e$.

\subsection{A divergence free vector field}  Then we define 
\begin{align}
\Phi(T,e)\defeq\log\bb\left(\delta^+(T,e),\delta^-(T,e),\alpha_0^+(T,e),\alpha_1^+(T,e)\right),\label{def:Phib}
\end{align}
where $\delta$ is the initial vertex of the edge, $\alpha_0$ and $\alpha_1$ are the vertices of $\Delta_0$ and $\Delta_1$ opposite to $e$. Then we have
\begin{proposition}{\sc[Divergence free]}
The vector field $\Phi$ is divergence free, moreover
\begin{align}
   \partial\Phi&=\ell(\partial S).
\end{align}
\end{proposition}
This proposition is a immediate consequence of the multiplicative cocycle properties of cross ratio.

\subsection{A special case: hyperbolic surface with one cusp} Instead of explaining in detail the not so exciting proof of the previous example, we insted give another related example  of the construction of the harmonic 1-form in the case of an hyperbolic surface with cusp which is perhaps more illuminating. 

Let us consider an embedded horosphere $H$ around the cusp. Every triangulation $T$ defines an ideal  hyperbolic triangulation of the surface. Then every oriented edge $e$ of $T$ defines a point $e_H$ in $H$, namely the (first) intersection of $H$ with the geodesic associated to $e$.   Let then $K$ be the collection of these points, and given $e_H\in K$, let $e^+H\in K$ be the point immediately on the right and $e^-_H$ be the point immediately on the left (for some auxiliary orientation of $H$). Let finally $I(e)$, be the interval on $H$ with extremities $e^\pm_H$ and containing $e_H$. Define
$$
\Phi(T,e)=\frac{1}{2}\frac{\operatorname{length}(I(e))}{\operatorname{length}(H)}.
$$
An elementary geometric construction shows you that

\begin{proposition}
The form $\Phi$ is harmonic:
$$
\Phi(\R(T,e))+\Phi(\L(T,e))=\Phi(T,e).
$$
Moreover $\partial\Phi=1$.
\end{proposition}

\subsection{Gaps and pairs of pants}

Let us move back again to the case of a general cross ratio and make the connection with \cite{Labourie:2009wi}.

Recall that a pair of pants is a regular homotopy class of immersion of the plane minus two points -- that we identify with $S^1\times ]-1,1[ \setminus\{1,0\}$ for later use -- in $S$. A pair of pants is alternatively described by  a triple $(\alpha,\beta,\gamma)$ of elements of $\grf$, defined up to conjugation so that $\alpha\cdotp\gamma\cdotp\beta=1$. The set of pairs of pants is then $\grf^3/\grf$ where the latter action is by conjugation.  A pair of pants is embedded if it can be represented by an embedding.

Then, following \cite{Labourie:2009wi}\footnote{We us a different convention for cross ratio}.

\begin{definition}
Let $\bb$ be a cross ratio. The {\em gap} of  the pair of pants $P=(\alpha,\beta,\gamma)$ with respect to the cross ratio $\bb$ is 
\begin{equation}
  \G_\bb(P)\defeq \log\left(\bb(\alpha^+,\alpha^-,\gamma^-,\beta^+)\right).
\end{equation}
\end{definition}
Observe that a  simple loop $\epsilon$ passing though $x_0$ define a pair of pants $P(\epsilon)$, associated to the embedding of $S^1\times ]-1,1[ \setminus\{1,0\}$ corresponding the the identification of a tubular neighbourhood of $\epsilon$ with $S^1\times ]-1,1[$.

We saw that if $T$ is a triangulation, then $\epsilon$ labels exactly one complementary   region  $f(\epsilon)$ of the tree $\mathcal T$. In particular, one can associate a gap to the harmonic function associated to ${\rm b}$.

The two notions of gap coincide.
\begin{proposition}
Let $\Phi$ be the associated divergence free vector field to a cross ratio $\bb$ as in Equation \eqref{def:Phib}. Then if $f$ is a complementary region,
\begin{equation} 
  \G_\Phi(f(\epsilon))=\G_\bb(P(\epsilon)).
\end{equation}

\end{proposition}

\begin{proof} This follows from Proposition \ref{spiral}. Indeed, by construction, if $f$ is a complementary region in  $\mc T$ an induction shows that 
$$
G^n_\Phi(f)={\rm b}(\alpha^+,\alpha^-,\gamma_n^-,\beta_n^+),
$$
where $\gamma_n$ is the simple arc associated to the edge $\R\cdotp\L^{n-1}\cdotp\R f$, and $\beta_n$ is the simple arc associated to the edge $\L\cdotp\R^{n-1}\cdotp\L f$. Then by Proposition \ref{spiral}, $\gamma_n$ converges to $\gamma=\partial^Rf$ and $\beta_n$ converges to $\beta=\partial^Lf$.
\end{proof}

Similarly in the case of hyperbolic surfaces with cusp one can get explicit formulas for the gaps in terms of lengths of simple curves bounding the pair of pants as in \cite{McShane:1998}, \cite{Mirzakhani:2007tt} and \cite{Labourie:2009wi}.

\section{The error term and the Birman--Series Theorem}\label{sec:BS}

Our Gap Formula for planar trees now reads, since every every complementary region is labelled by a simple loop passing though $x_0$, as
$$
\ell(\partial S)=\E(\Phi) + \sum_{\epsilon\in L(x_0)} \G_b(P(\epsilon)),
$$
where $L(x_0)$ is the set of simple loops up to homotopy passing though $x_0$.
To recover the  McShane's identity, we need to prove that the error term vanishes.

As in McShane's original proof and most of the subsequent proofs -- except notably the proof by Bowditch in \cite{Bowditch:1996} -- the vanishing of the error term relies on Birman--Series Theorem \cite{Birman:1985}: the closure of the reunion of the space of simple closed geodesics has Hausdorff dimension 1.

We will sketch this proof this only in the context of hyperbolic surface with one cusp and use the apparatus developed in paragraph \ref{sec:rv}.

In this context $\mathbb R/\partial\phi\mathbb Z$ is identified with the horosphere $H$ centered at $x_0$.  Then if $e$ is a simple loop passing though $x_0$ corresponding to an edge  in $S(n)$ of the tree $\mathcal T$, $Y_n(e)$ is the intersection of $e$ with the horosphere centered at $x_0$ -- normalized so that it has length one.  If $\p$ is an irrational path corresponding to an infinite simple geodesic $\gamma$ initiating at the cusp, we obtain that $X^-_\infty(\p)=X^+_\infty(\p)$ which corresponds to the intersection of $\gamma$ with the horosphere $H$. 

Then $\overline{X_\infty(\mc P\setminus\mc Q)}$ is included in the closure of the intersection of the horosphere  $H$ with the union of simple geodesics initiating in $x_0$. Thus, according to the Birman--Series Theorem, $\lambda(\overline{X_\infty(\mc P\setminus\mc Q)})=0$. Hence, by Theorem \ref{cor:ErrXinf}, the error term vanishes.

\section{Surfaces with more than one boundary}\label{sec:codBD} If $S$ has more than one puncture, $x_0, \ldots, x_{n-1}$, $n \ge 2$, we may still define $\Sigma=S\setminus\{x_0, \ldots, x_{n-1}\}$ and a triangulation of $\Sigma$ such that the set of vertices is $\{x_0, \ldots, x_{n-1}\}$. Again it is useful to endow $S$ with an auxiliary complete hyperbolic structure so that $x_0, \ldots, x_{n-1}$ correspond to cusps. We choose a distinguished cusp, say $x_0$, and this time, we let $L(x_0)$  be the set of oriented simple geodesic arcs starting from $x_0$ and terminating in another (not necessarily distinct) cusp. We can now  define $\mc S$ to be the space of pairs $(T,e)$ where $T$ is an ideal triangulation of $S$ and $e$ is an oriented edge of $T$ which is in  $L(x_0)$. 

The transformations $\F$, $\L$ and $\R$ are defined as before, with some slight modifications. To start with, note that $\F$ is defined  on the pairs $(T,e)$, where $e$ is an edge of $T$ but it is not necessary that  $e \in  L(x_0)$ ($e$ may originate from a different cusp).  On the other hand, the transformations $\L$ and $R$ may not be defined on $(T,e) \in \mc S$ for certain non-proper triangulations $T$ as we shall see later. Nonetheless, where $\L$ or $\R$ is defined on $(T,e)$, then $\L(T,e) \in \mc S$ and $\R (T,e) \in \mc S$. We now form the rooted tree as in the previous case, with the following modifications:

\begin{itemize}
	\item The root $v_0$ of the tree is an ideal triangulation $T_0$ of $S$ with vertices in the set $x_0, \ldots, x_{n-1}$.
	
	\item The neighbors of $v_0$ are the pairs  $(\mathring T, \mathring e) \in \mc S$ such that $(\mathring T, \mathring e)=\F(T_0,e)$ for some oriented edge $e$ of $T_0$ which is not necessarily in $L(x_0)$ (note however that we do require that $\mathring e \in L(x_0)$). In particular, it is possible for $v_0$ to be of valence 1 which occurs when  only one edge of $T_0$ (counted twice) is in $L(x_0)$. By construction, all the neighbors of $v_0$ are labeled by some element  $(T,e) \in \mc S$ and $e$ is adjacent to two distinct triangles of $T$.
	
	\item For a vertex $v$ labeled by $(T,e)$ at distance $n \ge 1$ from $v_0$, $\R(T,e)$ and $\L(T,e)$ are defined as before, provided that the corresponding flip moves are possible. Each admissible move results in a vertex at distance $n+1$ from $v_0$ adjacent to $v$.
	
	\item In the case of a vertex $(T,e)$ which for which $\R(T,e)$ is not defined, we call the vertex {\em right blocked}. Similarly, if $\L(T,e)$ is not defined, we call the vertex {\em left blocked}.

	\item A vertex $(T,e)$ is right blocked if $s^2(e)=\overline{s(e)}$  or equivalently, if $e$ bounds a punctured disk on the (see figure \ref{fig:Rightblocked}) . An analogous  statement holds for left blocked vertices .  If $S$ is not a thrice punctured sphere, then  a vertex $v \neq v_0$ cannot be both left blocked and right blocked. Hence the vertices of the rooted tree apart from the root have valence two or three.
	
\end{itemize}

\begin{figure}[h]
	\includegraphics[width=0.6\textwidth]{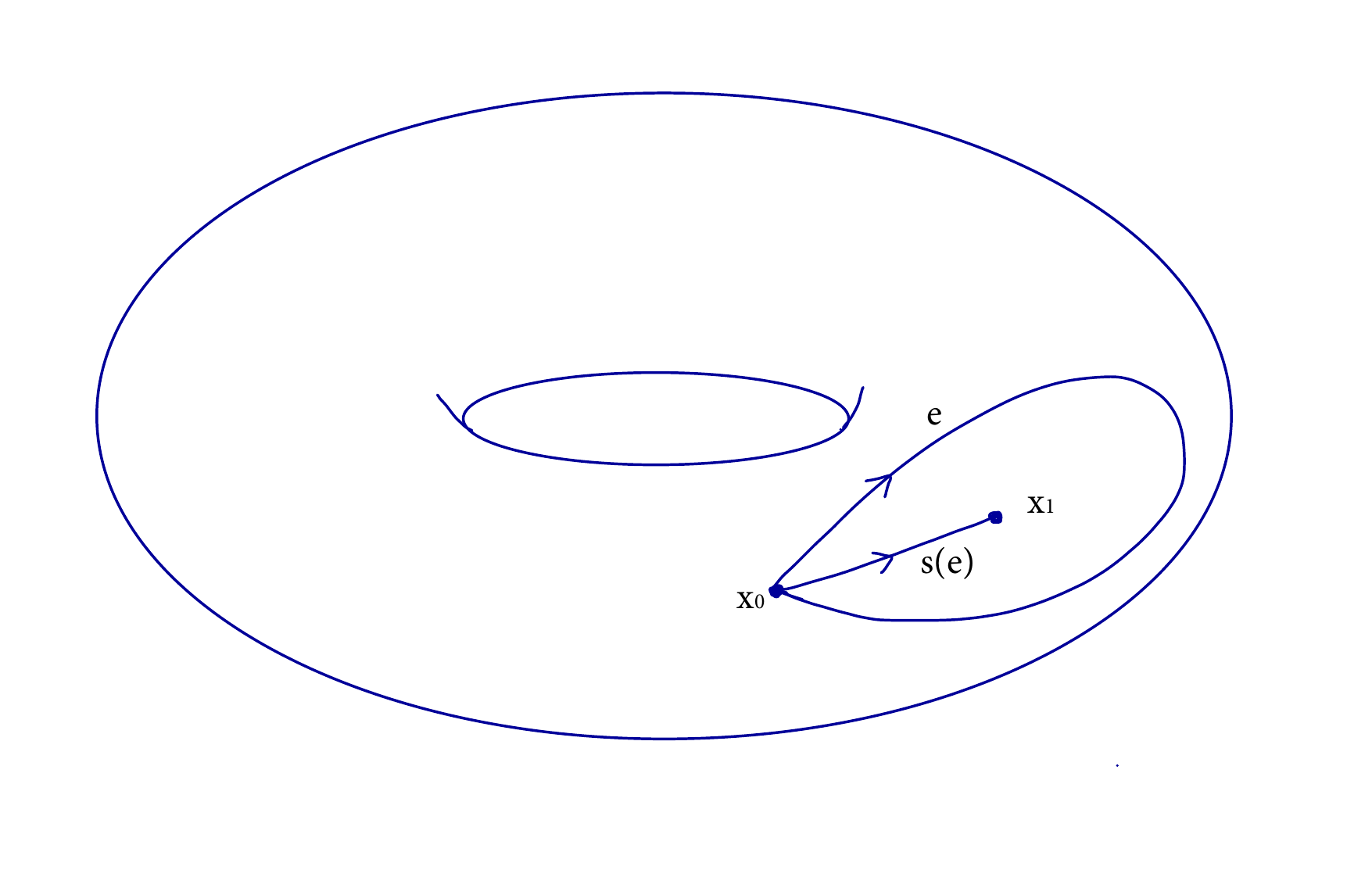}               
	\caption{An example of $(T,e)$ which is right-blocked. Here $S$ is a torus with two punctures $x_0, x_1$ and only the edges $e$ and $s(e)$ of the triangulation $T$ are shown.} \label{fig:Rightblocked}
\end{figure}

With this set-up, we have the following generalization of the topological coding theorem, with essentially the same proof:

\begin{theorem}{\sc [Topological coding for surfaces with more than one boundary]}
	Let $T_0$ be an ideal triangulation and $f$ an oriented topological embedded arc on $S$ starting from $x_0$ and ending in $\{x_0, \ldots, x_{n-1}\}$. Then, either $f$ is an edge of $T_0$ or there exists a unique oriented edge $e$ of $T_0$ and a unique element $g$ in $\Fd$ such that
	$\pi(g\cdotp\F(T_0,e))=f$. 
\end{theorem}

Similarly, the order preserving property holds. 

\noindent {\em Remark:} In the computation of the gap functions and the McShane identity, the ``end gaps'' correspond to the vertices which are left or right blocked.

\bibliographystyle{amsplain}


\end{document}